\documentclass{article}
\usepackage{amsmath,amssymb}

\newtheorem{theorem}{Theorem}[section]
\newtheorem{lemma}[theorem]{Lemma}
\newtheorem{proposition}[theorem]{Proposition}
\newtheorem{corollary}[theorem]{Corollary}

\newcommand{\id}{\mathrm{id}}
\newcommand{\supp}{\operatorname{supp}}

\newcommand{\Alt}{\operatorname{Alt}}
\newcommand{\Sym}{\operatorname{Sym}}
\newcommand{\FM}{\operatorname{FM}}

\newenvironment{proof}{\par\noindent{\bf Proof.}}{$\qed$\par\bigskip}
\newcommand{\qed}{\enspace\vrule  height6pt  width4pt  depth2pt}
\begin{document}

\title{Finitely Presented Monoids and Algebras defined by Permutation Relations of Abelian Type }

\author{ Ferran Ced\'{o}\footnote{
 Research partially supported by grants of MICIIN-FEDER (Spain)
MTM2008-06201-C02-01, Generalitat de Catalunya
2009 SGR 1389. }  \and Eric
Jespers\footnote{Research supported in part by
Onderzoeksraad of Vrije Universiteit Brussel and
Fonds voor Wetenschappelijk Onderzoek (Belgium).}
\and Georg Klein}
\date{}
\maketitle

\begin{abstract}
The class of finitely presented algebras over a
field $K$ with a set of generators $a_{1},\ldots
, a_{n}$ and defined by homogeneous relations of
the form
 $a_{1}a_{2}\cdots a_{n} =a_{\sigma (1)} a_{\sigma (2)} \cdots
 a_{\sigma (n)}$,
where $\sigma$ runs through an abelian subgroup $H$ of $\Sym_{n}$,
the symmetric group, is considered. It is proved that the Jacobson
radical of such algebras is zero. Also it is characterized when the
monoid $S_n(H)$, with the ``same'' presentation as the algebra, is
cancellative in terms of the stabilizer of $1$ and the stabilizer of
$n$ in $H$.  This work is a continuation of earlier work of Ced\'{o},
Jespers and Okni\'{n}ski.
\end{abstract}


\section{Introduction}\
In recent literature a lot of attention is given
to concrete classes of finitely presented
algebras $A$ over a field $K$ defined by
homogeneous semigroup relations, that is,
relations of the form $w=v$, where $w$ and $v$
are words of the same length in a generating set
of the algebra. In
\cite{alghomshort,altalgebra,alt4algebra} the
study of the following finitely presented
algebras over a field $K$ is initiated:
$$A=K\langle a_1,a_2,\dots ,a_n\mid a_1a_2\cdots
a_n= a_{\sigma (1)}a_{\sigma (2)}\cdots a_{\sigma (n)},\; \sigma\in
H\rangle , $$ where $H$ is a subset of the symmetric group $\Sym_n$
of degree $n$. Note that $A$ is the semigroup algebra $K[S_n(H)]$,
where $$S_n(H)=\langle a_1,a_2,\dots ,a_n\mid a_1a_2\cdots a_n=
a_{\sigma (1)}a_{\sigma (2)}\cdots a_{\sigma (n)},\; \sigma\in
H\rangle $$ is the monoid with the ``same'' presentation as the
algebra. In \cite{alghomshort}, the case being treated is that of
the cyclic  subgroup $H$ of $\Sym_n$ generated by $ \sigma = (1,2,
\ldots ,n)$. In \cite{altalgebra}  one deals with $H=\Alt_{n}$, the
alternating group, and in \cite{alt4algebra} the special and more
complicated case of $\Alt_{4}$ is handled. There are noteworthy
differences between these cases. In particular, the Jacobson radical
$J(K[S_n(\Alt_n)])$ of $K[S_{n}(\Alt_{n})]$ is zero only if $n$ is
even and $K$ has characteristic different from 2, and otherwise the
radical has been described, while $J(K[S_n(\langle \sigma
\rangle)])$ is always zero.  The latter is  a consequence of the
fact that $S_n(\langle \sigma \rangle)$ has a group of fractions
$G=S_n \langle a_1 \cdots a_n \rangle ^{-1} \cong F \times C $,
where $F=gr(a_1, \ldots ,a_{n-1})$ is a free group of rank $n-1$ and
$C=gr(a_1 \cdots a_n)$ is a cyclic infinite group.

Starting from the properties considered  in the
above mentioned papers, the aim  of this paper is
to investigate the properties of the algebra $
K[S_n(H)] $ for any abelian subgroup $H$ of
$\Sym_n$. In particular, we prove that
$J(K[S_n(H)])$ is always zero and we give
infinitely many examples of primitive ideals of
$K[S_n(H)]$ for $n\geq 3$ and $H$ abelian
subgroup of $\Sym_n$ such that
$(1,2,\dots ,n)\notin H$. Also we show that
$S_n(H)$ is a cancellative monoid if and only if
the stabilizers of $1$ and of $n$ in $H$ are
trivial subgroups of $H$.

\section{Preliminary results about $S_n(H)$}

The following two results display technical properties of
$S_n(H)$ which will be crucial for our investigations of $S_{n}(H)$
and $K[S_n(H)]$.

 Let $H$ be a subset of $\Sym_n$  and
$S=S_n(H)=\langle a_1, \ldots ,a_n \mid a_{1}\cdots a_{n}
=a_{\sigma(1)} \cdots a_{\sigma(n)},\; \sigma \in H \rangle$. We
denote by $z$ the element $z=a_1a_2\cdots a_n \in S$. Let $
\mathrm{FM}_n= \langle x_1, \ldots , x_n \rangle $ be the free
monoid of rank n and let $ \pi: \mathrm{FM}_n \rightarrow S $
be the unique morphism such that $ \pi(x_i) = a_i $ for all $ i = 1,
\ldots ,n $.

Let $w=x_{i_1}x_{i_2}\cdots x_{i_m}$ be a nontrivial word in the
free monoid $\mathrm{FM}_n$. Let $1\leq p,q\leq m$ and $r,s$ be
nonnegative integers such that $p+r,q+s\leq m$. We say that the
subwords $w_1=x_{i_p}x_{i_{p+1}}\cdots x_{i_{p+r}}$ and
$w_2=x_{i_q}x_{i_{q+1}}\cdots x_{i_{q+s}}$ overlap in $w$ if either
$p\leq q\leq p+r$ or $q\leq p\leq q+s$. For example, in the word
$x_2x_2x_3x_1x_4$ the subwords $x_2x_3x_1$ and $x_3x_1x_4$ overlap
and the subwords $x_2x_2$ and $x_3x_1$ do not overlap. If $p\leq
q\leq p+r\leq q+s$, then we say that the length of the overlap
between the subwords $w_1$ and $w_2$ is $p+r-q+1$. If $p\leq q\leq
q+s\leq p+r$, then we say that the length of the overlap between the
subwords $w_1$ and $w_2$ is $s+1$. For example, the length of the
overlap between the subwords $x_2x_3x_1$ and $x_3x_1x_4$ in
$x_2x_2x_3x_1x_4$ is $2$.

We denote by $|w|$ the length of the word $w\in \mathrm{FM}_n$.

 For, $1\leq i \leq n$  and $H$ any subgroup
of $\Sym_{n}$, we denote by $H_{i}$ the stabilizer of $i$ in $H$.
Thus $H_{i}=\{ \sigma \in H \mid \sigma (i)=i\}$.  The identity map
in $\Sym_{n}$ we denote by $\id$.

\begin{lemma}\label{L1}
Let H be an abelian subgroup of $\Sym_n$. Let
$w_1,w_2, w_{1,1}$, $w_{1,2},\; w_{1,3}$,
$w_{1,1}',\;  w_{1,2}',\; w_{1,3}',\;  w_{2,2},\;
w_{2,1}',\;  w_{2,2}',\; w_{2,3}'\in
\mathrm{FM}_n$  be such that

\begin{equation} \label{xeq1}
\begin{split}
& w_1 = w_{1,1}w_{1,2}w_{1,3}= w_{1,1}'w_{1,2}'w_{1,3}',\\
& w_2 = w_{1,1}'w_{2,2}w_{1,3}'=w_{2,1}'w_{2,2}'w_{2,3}'
\end{split}
\end{equation}
and
$\pi(w_{1,2})=\pi(w_{1,2}')=\pi(w_{2,2})=\pi(w_{2,2}')=
z$.
\begin{itemize}
\item[(i)]
If $w_{1,2}$ and $w_{1,2}'$ overlap in $w_1$, $w_{2,2}$ and
$w_{2,2}'$ overlap in $w_2$, $|w_{1,3}|,|w_{2,3}'|<|w_{1,3}'|$ and
$H_{n}=\{ \id\}$, then $w_{1}=w_{2}$.
\item[(ii)]
If $w_{1,2}$ and $w_{1,2}'$ overlap in $w_1$, $w_{2,2}$ and
$w_{2,2}'$ overlap in $w_2$, $|w_{1,1}|,|w_{2,1}'|<|w_{1,1}'|$ and
$H_{1}=\{ \id\}$, then $w_{1}=w_{2}$.
\end{itemize}

\end{lemma}

\begin{proof}
  $(i)$ Since $w_{1,2}$ and $w_{1,2}'$
overlap in $w_1$ and $|w_{1,3}|<|w_{1,3}'|$, we
have that $0<|w_{1,3}'|-|w_{1,3}|<n$. Since
$w_{2,2}$ and $w_{2,2}'$ overlap in $w_2$ and
$|w_{2,3}'|<|w_{1,3}'|$, we have that
$0<|w_{1,3}'|-|w_{2,3}'|<n$. Thus, there exist
$u,v\in \mathrm{FM}_n$ such  that
\begin{eqnarray} \label{first*}
w_{1,2}'w_{1,3}'=uw_{1,2}w_{1,3}&\quad\mbox{and}\quad
& w_{2,2}w_{1,3}'=vw_{2,2}'w_{2,3}'.
\end{eqnarray}

Suppose that the length of the overlap between
$w_{1,2}$ and $w_{1,2}'$ in $w_1$ is $i$. Then
there exist $\sigma_1, \sigma_2 \in H$ such that
$$ uw_{1,2} = x_{\sigma_1(1)} \cdots
x_{\sigma_1(n)} x_{\sigma_2(i+1)} \cdots x_{\sigma_2(n)}, $$
with
\begin{eqnarray}\label{first**}
\sigma_1(n-i+1)= \sigma_2(1),\;  \sigma_1(n-i+2)= \sigma_2(2),\;
\dots, \; \sigma_1(n)=\sigma_2(i), \end{eqnarray}
 and $
x_{\sigma_2(i+1)} \cdots x_{\sigma_2(n)}w_{1,3}=w_{1,3}'$.  Note
that $|u|=n-i$.

Suppose that the length of the overlap between
$w_{2,2}$ and $w_{2,2}'$ in $w_2$ is $j$. Then
there exist $\tau_1, \tau_2 \in H$ such that $$
vw_{2,2}' = x_{\tau_1(1)} \cdots x_{\tau_1(n)}
x_{\tau_2(j+1)} \cdots x_{\tau_2(n)}, $$
with
\begin{eqnarray} \label{second*}
\tau_1(n-j+1)= \tau_2(1),\;  \tau_1(n-j+2)=
\tau_2(2),\; \dots, \tau_1(n)=\tau_2(j),
\end{eqnarray}
 and $ x_{\tau_2(j+1)} \cdots
x_{\tau_2(n)}w_{2,3}'=w_{1,3}'$.

 From  (\ref{first*}) we obtain that
$|u|+|w_{1,2}|+|w_{1,3}|=|w_{1,2}'|+ |
w_{1,3}'|$. Hence, $|u|=|w_{1,3}'|-|w_{1,3}|$ and
therefore $0<i=n-(|w_{1,3}'|-|w_{1,3}|)<n$.
Similarly, $0<j=n-(|w_{1,3}'|-|w_{2,3}'|)<n$.

Suppose that $1 \leq i \leq j < n$.  In this
case, we have from (\ref{second*}) that
$\tau_2(i)=\tau_1(n-j+i)$ and, since $\sigma_2(i)
= \sigma_1(n)$, we get $$ i = \sigma_2^{-1}
\sigma_1(n)= \tau_2^{-1} \tau_1(n-j+i). $$ Since
$H$ is abelian, we  thus obtain
$$\tau_2(n)=\tau_1 \sigma_1^{-1}\sigma_2
(n-j+i).$$  As  $w_{1,3}'=x_{\sigma_2(i+1)}
\cdots x_{\sigma_2(n)}w_{1,3}=x_{\tau_2(j+1)}
\cdots x_{\tau_2(n)}w_{2,3}' $ and $i\leq j$, we
have that $\tau_{2}(n) =\sigma_2(n-j+i)$. Hence
$$\tau_2(n)=\tau_1 \sigma_1^{-1} \tau_2(n).$$
 Because, by assumption,  $H_{n}=\{
\id\}$, we get that $\tau_{2}=\tau_{1}\sigma_{1}^{-1}\tau_{2}$ and
thus $ \tau_1 = \sigma_1$. Therefore
\begin{eqnarray*}
w_1&=&w_{1,1}'w_{1,2}'w_{1,3}'=w_{1,1}'x_{\sigma_1(1)} \cdots x_{\sigma_1(n)}w_{1,3}'\\
&=&w_{1,1}'x_{\tau_1(1)} \cdots x_{\tau_1(n)}w_{1,3}'=w_{1,1}'w_{2,2}w_{1,3}'\\
&=&w_2.
\end{eqnarray*}

Suppose now that $1 \leq j < i < n$. In this case,  from
(\ref{first**}) we have $\sigma_2(j)=\sigma_1(n-i+j)$ and, since
$\tau_2(j) = \tau_1(n)$, we get $$ j = \sigma_2^{-1}
\sigma_1(n-i+j)= \tau_2^{-1} \tau_1(n). $$  As $H$ is abelian,
we have
$$\sigma_2(n)=\sigma_1 \tau_1^{-1}\tau_2
(n-i+j).$$  Because
$w_{1,3}'=x_{\sigma_2(i+1)} \cdots
x_{\sigma_2(n)}w_{1,3}=x_{\tau_2(j+1)} \cdots
x_{\tau_2(n)}w_{2,3}' $ and $i> j$, we have that
$\sigma_{2}(n) =\tau_2(n-i+j)$. Hence
$$\sigma_2(n)=\sigma_1 \tau_1^{-1}
\sigma_2(n).$$  Since, by assumption, $H_{n}=\{ \id\}$, we obtain
that $ \tau_1 = \sigma_1$. Thus, also in this case, $w_1=w_2$.
Therefore part $(i)$ follows.

 Part $(ii)$ of the lemma follows by symmetry. Or alternatively,
the opposite monoid $S^{opp}$ is a monoid of the same type as $S$,
where we replace the element $z=a_{1}\cdots a_{n}$ by the element
$a_{n}\cdots a_{1}$ . Hence, if $H_{1}=\{ \id\}$ then  $(i)$ holds
for $S^{opp}$ and thus $(ii)$ holds for $S$. Recall that as a set
$S^{opp}$ is $S$ but multiplication $\cdot$ in $S^{opp}$ is defined
by $s_{1}\cdot s_{2}=s_{2}s_{1}$, where the latter is the product in
$S$.
\end{proof}

 Let \begin{eqnarray} \label{third*}
 A =
\{x_{\sigma(n-1)}x_{\sigma(n)} \mid \sigma \in H
\} &\quad  \mbox{and} \quad &  \tilde{A} =
\{x_{\sigma(1)}x_{\sigma(2)} \mid \sigma \in H
\}.
\end{eqnarray}

\begin{lemma}\label{L2}
 Let $H$ be an abelian  subgroup of $\Sym_n$
such that $(1,2,\dots ,n)\not\in H$.
\begin{itemize}
\item[(i)] If $H_{n}=\{ \id\}$ and if
 $w\in \mathrm{FM}_n$ is
such that $\pi(wx_ix_j)\in Sz$ for some $1 \leq
i,j \leq n$, then $x_ix_j \in A$.
\item[(ii)] If $H_{1}=\{ \id\}$ and if
$v\in \mathrm{FM}_n$ is such that
$\pi(x_ix_jv)\in zS$ for some $1 \leq i,j \leq
n$, then $x_ix_j \in \tilde{A}$.
\end{itemize}

\end{lemma}

\begin{proof}
 $(i)$ Suppose that the result is false. Let
$w \in \mathrm{FM}_n$ and $1 \leq i_0,j_0 \leq n$
such that $\pi(wx_{i_0}x_{j_0})\in Sz$ and
$x_{i_0}x_{j_0} \notin A$. Hence there exist
$w_{0,1}$, $w_{0,2} \in \mathrm{FM}_n$ such that
$\pi(w_{0,1}w_{0,2})=\pi(wx_{i_0}x_{j_0})$ and
$\pi(w_{0,2})=z$. Thus there exist
$w_0,w_1,\ldots,w_t \in \mathrm{FM}_n$ such that
$w_0=w_{0,1}w_{0,2}$, $w_t=wx_{i_0}x_{j_0}$,

\begin{equation} \label{xeq4}
\begin{split}
& \pi(w_0)=\pi(w_1)=\ldots=\pi(w_t)=\pi(wx_{i_0}x_{j_0})
\end{split}
\end{equation}
and there exist $w_{i,2}\in  \mathrm{FM}_n$,
for $i=1,2,\ldots,t$, and $w_{j,1}', w_{j,2}', w_{j,3}' \in \mathrm{FM}_n$, for
$j=0,1,\ldots,t-1$, such that

\begin{equation} \label{xeq5}
\begin{split}
& w_0 = w_{0,1}w_{0,2} = w_{0,1}'w_{0,2}'w_{0,3}', \\
& w_k = w_{k-1,1}'w_{k,2}w_{k-1,3}'=w_{k,1}'w_{k,2}'w_{k,3}' \text{ , for } k=1,2,\ldots,t-1,\\
& w_t = w_{t-1,1}'w_{t,2}w_{t-1,3}',  \\
& \pi(w_{i,2})=\pi(w_{j,2}')=z,
\end{split}
\end{equation}
for all $i=0,1,2,\ldots,t$ and for all $j=0,1,\ldots,t-1$.

We choose a sequence $w_0,w_1,\dots ,w_t$,
with a decomposition (\ref{xeq5}), such that
$w_k=x_{k(1)}x_{k(2)} \cdots x_{k(n+m)}$, for all $k=0,1,\dots,t$, and
$x_{t(n+m-1)}x_{t(n+m)}\notin A$,
with $t$  minimal.  By the minimality of $t$,
we have that

\begin{equation} \label{xeq6}
\begin{split}
& x_{k(n+m-1)}x_{k(n+m)} \in A \text{ for all }  k=1,2,\ldots,t-1 ,\\
& x_{t(n+m-1)}x_{t(n+m)} \notin A ,\\
& \pi(x_{k(m+1)}x_{k(m+2)}\cdots x_{k(m+n)}) \neq z \text{ for all } k=1,2,\ldots,t  .
\end{split}
\end{equation}
Since $\pi(w_{0,2})=z$ and $\pi(x_{1(m+1)}x_{1(m+2)} \cdots x_{1(m+n)}) \neq z$, we have that
$1 \leq |w_{0,3}'| < n$. Note that then
the subwords $w_{0,2}$ and $w_{0,2}'$ overlap in $w_0$.

Suppose that $t=1$. In this case, since
$\pi(x_{0(m+1)}x_{0(m+2)}\cdots x_{0(m+n)})=z$
and $ x_{1(n+m-1)}x_{1(n+m)} \notin A$, we have
that $|w_{0,3}'|=1$. Hence $w_{0,3}'=x_{0(m+n)}$
and $$x_{0(m+n)}w_{0,2}=w_{0,2}'x_{0(m+n)}.$$
Thus there exist $\tau_1,\tau_2\in H$ such that
$$x_{0(m+n)}x_{\tau_1(1)}x_{\tau_1(2)}\cdots
x_{\tau_1(n)}=x_{\tau_2(1)}x_{\tau_2(2)}\cdots
x_{\tau_2(n)}x_{0(m+n)}.$$ Therefore
$$\tau_1(1)=\tau_2(2),\tau_1(2)=\tau_2(3),\dots
,\tau_1(n-1)=\tau_2(n),
\tau_1(n)=0(m+n)=\tau_2(1).$$ Hence
$\tau_2^{-1}\tau_1=(1,2,\dots ,n)$,  in
contradiction with the assumption. Therefore
$t>1$.
\medskip

{\em Claim:} $0<|w_{j,3}'|-|w_{j-1,3}'|<n$ for all $j=1,2,\ldots,t-1$.
\medskip

Suppose that the claim is false. Let $r \in \{1, \dots ,t-1 \}$ be the smallest value
such that either $|w_{r,3}'| \leq |w_{r-1,3}'|$ or  $ |w_{r,3}'|-|w_{r-1,3}'|\geq n$.

Suppose that $|w_{r,3}'| \leq |w_{r-1,3}'|$.

If $|w_{r,3}'|=|w_{r-1,3}'|$, then, since $w_r=w_{r-1,1}'w_{r,2}w_{r-1,3}'=w_{r,1}'w_{r,2}'w_{r,3}'$, we have that
$w_{r-1,3}'=w_{r,3}'$ and $w_{r-1,1}'=w_{r,1}'$. Hence, in this case, the sequence
$w_0,w_1, \dots ,w_{r-1},w_{r+1},w_{r+2},\dots , w_t$ is a shorter sequence with a decomposition of type
(\ref{xeq5}), in contradiction
with the minimality of $t$.
Hence $|w_{r,3}'| < |w_{r-1,3}'|$.

If $|w_{r-1,3}'|-|w_{r,3}'|<n$, then   $w_{r,2}$ and $w_{r,2}'$ overlap in $w_r$.
Now we have

\begin{equation} \label{xeq8}
\begin{split}
 &w_{r-1} = w_{r-2,1}'w_{r-1,2}w_{r-2,3}' = w_{r-1,1}'w_{r-1,2}'w_{r-1,3}', \\
 &w_r = w_{r-1,1}'w_{r,2}w_{r-1,3}' = w_{r,1}'w_{r,2}'w_{r,3}',
\end{split}
\end{equation}
(here, if $r=1$, we  agree that
$w_{r-2,1}'=w_{0,1}$ and $w_{r-2,3}' =1$). Since
$0<|w_{r-1,3}'|-|w_{r-2,3}'|<n$, we have that
$w_{r-1,2}$ and $w_{r-1,2}'$ overlap in
$w_{r-1}$. Since $w_{r,2}$ and $w_{r,2}'$ also
overlap in $w_r$ and $
|w_{r-2,3}'|,|w_{r,3}'|<|w_{r-1,3}'|$, by
Lemma~\ref{L1}, we have that $w_{r-1}=w_r$. Now
the sequence $w_0,w_1,\dots
,w_{r-1},w_{r+1},w_{r+2},\dots ,w_t$ with the
decomposition as in (\ref{xeq5}), except for
$w_{r-1}=w_{r-2,1}'w_{r-1,2}w_{r-2,3}' =
w_{r,1}'w_{r,2}'w_{r,3}'$, is a shorter sequence
with a decomposition of type (\ref{xeq5}), in
contradiction with the minimality of $t$. Hence
$|w_{r-1,3}'|-|w_{r,3}'| \geq n$.

Since $|w_{0,3}'|<n$, we have that $r>1$.
Let $l\in\{ 0,1,\dots ,r-2\}$ be the smallest value
such that $|w_{l+1,3}'|-|w_{r,3}'|\geq n$. Since $0<|w_{j,3}'|-|w_{j-1,3}'|<n$, for all $j=1,2,\dots ,r-1$, there
exists $u\in \mathrm{FM}_n$ such that
$w_{l+1,3}'=uw_{r,2}'w_{r,3}'$. Now we have

\begin{equation} \label{xeq10a}
\begin{split}
 &w_{l} = w_{l-1,1}'w_{l,2}w_{l-1,3}' = w_{l,1}'w_{l,2}'w_{l,3}', \\
 &w_{l+1} = w_{l,1}'w_{l+1,2}w_{l,3}' = (w_{l+1,1}'w_{l+1,2}'u)w_{r,2}'w_{r,3}',
\end{split}
\end{equation}
(here, if $l=0$, we put $w_{l-1,1}'=w_{0,1}$ and
$w_{l-1,3}'=1$). Since
$0<|w_{l,3}'|-|w_{l-1,3}'|<n$, we have that
$w_{l,2}$ and $w_{l,2}'$ overlap in $w_{l}$.
Since $w_{l+1,2}$ and $w_{l+1,2}'$ overlap in
$w_{l+1}$, we have that $|w_{l,3}'|>|w_{r,3}'|$.
By the choice of $l$, $|w_{l,3}'|-|w_{r,3}'|<n$.
Hence $w_{l+1,2}$ and $w_{r,2}'$ overlap in
$w_{l+1}$. Thus, applying Lemma~\ref{L1} to
(\ref{xeq10a}), we obtain that $w_{l}=w_{l+1}$.
Now the sequence $w_0,w_1,\dots
,w_{l},w_{l+2},w_{l+3},\dots ,w_t$ with the
decomposition as in (\ref{xeq5}), except for
$w_{l}=w_{l-1,1}'w_{l,2}w_{l-1,3}' =
w_{l+1,1}'w_{l+1,2}'w_{l+1,3}'$, is a shorter
sequence with a decomposition of type
(\ref{xeq5}), in contradiction with the
minimality of $t$. Hence
$|w_{r,3}'|>|w_{r-1,3}'|$. Therefore
$|w_{r,3}'|-|w_{r-1,3}'|\geq n$.

Since
$w_r=w_{r-1,1}'w_{r,2}w_{r-1,3}'=w_{r,1}'w_{r,2}'w_{r,3}'$,
we  also have that $|w_{r-1,1}'|
=|w_{r,3}'|-|w_{r-1,3}'| +|w_{r,1}'| \geq
|w_{r,1}'|+n$ and  therefore
\begin{eqnarray} \label{fourth*}
w_{r-1,1}'\in w_{r,1}'w_{r,2}'\mathrm{FM}_n.
\end{eqnarray}
Since $0<|w_{0,3}'|<|w_{1,3}'|<\dots <
|w_{r-1,3}'|$ and
$$w_k=w_{k-1,1}'w_{k,2}w_{k-1,3}'=w_{k,1}'w_{k,2}'w_{k,3}',$$
for all $k=1,2,\dots , r-1$, we have  that $
|w_{k-1,1}'|>|w_{k,1}'|$ and thus
$w_{k-1,1}'\in w_{k,1}'\mathrm{FM}_n$, for all
$k=1,2,\dots , r-1$.  Hence, for all
$k=0,1,\ldots , r-1$, $w_{k,1}'\in w_{r-1,1}'
\mathrm{FM}_{n}$, and therefore from
(\ref{fourth*}), there exist $v_k'\in
\mathrm{FM}_n$ such that
$$w_{k,1}'=w_{r,1}'w_{r,2}'v_k',$$ for all
$k=0,1,2,\dots , r-1$.

Consider the following sequence:

\begin{equation} \label{xeq14}
\begin{split}
 & w_0=w_{0,1}w_{0,2}=w_{r,1}'w_{r,2}'(v_0'w_{0,2}'w_{0,3}'), \\
 &w_1'=w_{r,1}'w_{r+1,2}(v_0'w_{0,2}'w_{0,3}')=(w_{r,1}'w_{r+1,2}v_0')w_{0,2}'w_{0,3}', \\
 &w_2'=(w_{r,1}'w_{r+1,2}v_0')w_{1,2}w_{0,3}'=(w_{r,1}'w_{r+1,2}v_1')w_{1,2}'w_{1,3}', \\
 &\hspace{0.2 cm} \vdots   \\
 &w_{r-1}'=(w_{r,1}'w_{r+1,2}v_{r-3}')w_{r-2,2}w_{r-3,3}'=(w_{r,1}'w_{r+1,2}v_{r-2}')w_{r-2,2}'w_{r-2,3}', \\
 &w_r'=(w_{r,1}'w_{r+1,2}v_{r-2}')w_{r-1,2}w_{r-2,3}'=(w_{r,1}'w_{r+1,2}v_{r-1}')w_{r-1,2}'w_{r-1,3}', \\
 &w_{r+1}'=(w_{r,1}'w_{r+1,2}v_{r-1}')w_{r,2}w_{r-1,3}'.
\end{split}
\end{equation}
Since
$w_r=w_{r-1,1}'w_{r,2}w_{r-1,3}'=w_{r,1}'w_{r,2}'w_{r,3}'$
and $w_{r-1,1}'=w_{r,1}'w_{r,2}'v_{r-1}'$, we
have that $w_{r,3}'=v_{r-1}'w_{r,2}w_{r-1,3}'$.
Hence
$$w_{r+1}=w_{r,1}'w_{r+1,2}w_{r,3}'=w_{r,1}'w_{r+1,2}v_{r-1}'w_{r,2}w_{r-1,3}'=w_{r+1}'.$$
 As $|w_{0,2}|=n \leq |v_{0}'w_{0,2}'w_{0,3}'|$ we know that
$w_1'\in \mathrm{FM}_nw_{0,2}$ and $\pi (w_{1}')\in Sz$. Now the
the sequence $w_1',w_2',\dots ,w_r',w_{r+1},w_{r+2},\dots ,w_t$,
with the decomposition (\ref{xeq14}) for $w_1',w_2',\dots ,w_r'$,
the decomposition
$$w_{r+1}=(w_{r,1}'w_{r+1,2}v_{r-1}')w_{r,2}w_{r-1,3}'=w_{r+1,1}'w_{r+1,2}'w_{r+1,3}',$$
 for $w_{r+1}$ and the decomposition
(\ref{xeq5}) for $w_{r+2},\dots ,w_t$, is a
shorter sequence with a decomposition of type
(\ref{xeq5}).  This is in contradiction with
the minimality of $t$. Therefore the claim
follows.

 So we have that $0<|w_{0,3}'|< |w_{1,3}'| <
\cdots < |w_{t-1,3}'|$. Because $w_{t-1} =
w_{t-1,1}'w_{t-1,2}'w_{t-1,3}'$ and
$w_{t}=w_{t-1,1}'w_{t,2}w_{t-1,3}'$, we obtain
that
$$x_{(t-1)(m+n-1)}x_{(t-1)(m+n)}=x_{(t)(m+n-1)}x_{(t)(m+n)}.$$
 However, by $(\ref{xeq6})$, we know that
$x_{(t-1)(m+n-1)}x_{(t-1)(m+n)}\in A$, while we
also have that $x_{t(m+n-1)} x_{t(m+n)}\not\in
A$, a contradiction.   Therefore part $(i)$
follows.

 Part $(ii)$ follows by considering part (i)
to the opposite monoid $S^{opp}$.
\end{proof}

\section{Cancellativity of $S_n(H)$}

 Let $H$ be a subgroup of $\Sym_{n}$. For $S_{n}(H)$ to be
cancellative, a necessary condition is that $H_{1}=H_{n}=\{ \id\}$.
Hence, $S_{n}(\Sym_{n})$ with $n\geq 3$, and $S_{n}(\Alt_{n})$ with
$n\geq 4$ are not cancellative. In \cite{alghomshort} it is shown
that $S_{n}(\langle (1,2,\ldots , n) \rangle )$ is cancellative and
has a group of fractions. We now prove that for $H$ abelian,
$H_{1}=H_{n}=\{ \id\}$ also is a sufficient condition for $S_{n}(H)$
to be cancellative. Note that if also $n\geq 3$ and $(1,2,\ldots ,
n)\not\in H $ then from Lemma~\ref{L2}, $S_{n}(H)z\cap
S_{n}(H)x_{1}^{2}=\emptyset$ and $x_{1}^{2}S_{n}(H)\cap zS_{n}(H)
=\emptyset$. Therefore, for such $H$, $S_{n}(H)$ does not have a
group of fractions.

\begin{theorem}\label{T3}
Let H be an abelian subgroup of $Sym_n$, and let $S=S_n(H)$. Then
$S$ is cancellative if and only if $H_1=H_n=\{ \id\}$.
\end{theorem}

\begin{proof}

 That the conditions $H_1=H_n=\{ \id\}$ are necessary has been
mentioned above. For the converse, assume $H_1=H_n=\{ \id\}$. We
shall prove that $S$ is right cancellative. Then, as mentioned
before, working with $S^{opp}$, the left cancellativity will follow.
If $(1,2,\dots ,n)\in H$ then $H$ is transitive. As $H_1$ is
trivial, we then get that  $$n=|\{\sigma (1)\mid \sigma\in
H\}|=|H|/|H_1|=|H|.$$ Hence $H=\langle (1,2,\dots ,n)\rangle$.
Therefore, as mentioned above (\cite[Theorem 2.2]{alghomshort}) $S$
is cancellative.

Thus we may assume that $(1,2,\dots ,n)\notin
H$.

Suppose that $S$ is not right cancellative. Then
there exist $a,b\in S$ and $1\leq i\leq n$ such
that $a\neq b$ and $aa_i=ba_i$. Let $u,v\in
\FM_n$  be such that $\pi(u)=a$ and
$\pi(v)=b$. Since $aa_i=ba_i$, there exist
$w_0,w_1,\dots ,w_t\in\FM_n$ such that
$w_0=ux_i$, $w_t=vx_i$, $$\pi(w_0)=\pi(w_1)=\dots
=\pi(w_t)$$ and there exist $w_{i,2}\in \FM_n$,
for $i=1,\dots ,t$, and
$w'_{j,1},w'_{j,2},w'_{j,3}\in \FM_n$, for
$j=0,1,\dots ,t-1$, such that
\begin{equation} \label{xeq5bis}
\begin{split}
& w_0 = w_{0,1}'w_{0,2}'w_{0,3}', \\ & w_k =
w_{k-1,1}'w_{k,2}w_{k-1,3}'=w_{k,1}'w_{k,2}'w_{k,3}'
\text{ , for } k=1,2,\ldots,t-1,\\ & w_t =
w_{t-1,1}'w_{t,2}w_{t-1,3}',  \\ &
\pi(w_{i,2})=\pi(w_{j,2}')=z,
\end{split}
\end{equation}
for all $i=0,1,2,\ldots,t$ and for all
$j=0,1,\ldots,t-1$.

We choose a sequence $w_0,w_1,\dots ,w_t$, with a
decomposition (\ref{xeq5bis}), such that
$w_k=x_{k(1)}x_{k(2)} \cdots x_{k(m)}$, for all
$k=0,1,\dots,t$, $$\pi(x_{0(1)}\cdots
x_{0(m-1)})\neq\pi(x_{t(1)}\cdots
x_{t(m-1)})\quad \mbox{and}\quad
x_{0(m)}=x_{t(m)},$$ with $t$ minimal.

By the minimality of $t$, we have that
$w'_{0,3}=1$.

Suppose that $t=1$. In this case,
$w_0=w_{0,1}'w_{0,2}'$ and $w_1=w_{0,1}'w_{1,2}$,
with $w_{0,2}'\neq w_{1,2}$. Since
$\pi(w_{0,2}')=\pi(w_{1,2})=z$, there exist
$\sigma,\tau\in H$ such that $\sigma\neq \tau$,
$$w_{0,2}'=x_{\sigma(1)}\cdots
x_{\sigma(n)}\quad\mbox{and}\quad w_{1,2}=x_{\tau(1)}\cdots
x_{\tau(n)}.$$ Since
$x_{\sigma(n)}=x_{0(m)}=x_{t(m)}=x_{1(m)}=x_{\tau(n)}$, we have that
$\id \neq \sigma^{-1}\tau\in H_n$. But  this yields a contradiction
as, by assumption,  $H_n=\{ \id\}$. Therefore $t>1$.
\medskip

{\em Claim:} $0<|w_{j,3}'|-|w_{j-1,3}'|<n$ for
all $j=1,2,\ldots,t-1$.
\medskip

Suppose that the claim is false. Let $r \in \{1,
\dots ,t-1 \}$ be the smallest value such that
either $|w_{r,3}'| \leq |w_{r-1,3}'|$ or  $
|w_{r,3}'|-|w_{r-1,3}'|\geq n$.

Suppose that $|w_{r,3}'| \leq |w_{r-1,3}'|$.

If $|w_{r,3}'|=|w_{r-1,3}'|$, then, since
$w_r=w_{r-1,1}'w_{r,2}w_{r-1,3}'=w_{r,1}'w_{r,2}'w_{r,3}'$,
we have that $w_{r-1,3}'=w_{r,3}'$ and
$w_{r-1,1}'=w_{r,1}'$. Hence, in this case, the
sequence $w_0,w_1, \dots
,w_{r-1},w_{r+1},w_{r+2},\dots , w_t$ is a
shorter sequence with a decomposition of type
(\ref{xeq5bis}), in contradiction with the
minimality of $t$. Hence $|w_{r,3}'| <
|w_{r-1,3}'|$. Since $w'_{0,3}=1$, we have that
$r>1$.

If $|w_{r-1,3}'|-|w_{r,3}'|<n$, then   $w_{r,2}$
and $w_{r,2}'$ overlap in $w_r$. Now we have

\begin{equation} \label{xeq8aa}
\begin{split}
 &w_{r-1} = w_{r-2,1}'w_{r-1,2}w_{r-2,3}' = w_{r-1,1}'w_{r-1,2}'w_{r-1,3}', \\
 &w_r = w_{r-1,1}'w_{r,2}w_{r-1,3}' = w_{r,1}'w_{r,2}'w_{r,3}'.
\end{split}
\end{equation}
Since $0<|w_{r-1,3}'|-|w_{r-2,3}'|<n$, we have
that $w_{r-1,2}$ and $w_{r-1,2}'$ overlap in
$w_{r-1}$. Since $w_{r,2}$ and $w_{r,2}'$ also
overlap in $w_r$ and $
|w_{r-2,3}'|,|w_{r,3}'|<|w_{r-1,3}'|$,  we
obtain from Lemma~\ref{L1}  that $w_{r-1}=w_r$.
Now the sequence $w_0,w_1,\dots
,w_{r-1},w_{r+1}$, $w_{r+2},\dots ,w_t$ with the
decomposition as in (\ref{xeq5bis}), except for
$$w_{r-1}=w_{r-2,1}'w_{r-1,2}w_{r-2,3}' =
w_{r,1}'w_{r,2}'w_{r,3}',$$ is a shorter sequence
with a decomposition of type (\ref{xeq5bis}), in
contradiction with the minimality of $t$. Hence
$|w_{r-1,3}'|-|w_{r,3}'| \geq n$.

 Recall that  $w_{0,3}'=1$. Thus we have that $r>1$. Let
$l\in\{ 0,1,\dots ,r-2\}$ be the smallest value such that
$|w_{l+1,3}'|-|w_{r,3}'|\geq n$. Since
$0<|w_{j,3}'|-|w_{j-1,3}'|<n$, for all $j=1,2,\dots ,r-1$, there
exists $u\in \mathrm{FM}_n$ such that
$w_{l+1,3}'=uw_{r,2}'w_{r,3}'$. Since $0<|w'_{1,3}|-|w'_{0,3}|<n$
and $w'_{0,3}=1$, we have that $l>0$. Now, we have

\begin{equation} \label{xeq10aa}
\begin{split}
 &w_{l} = w_{l-1,1}'w_{l,2}w_{l-1,3}' = w_{l,1}'w_{l,2}'w_{l,3}', \\
 &w_{l+1} = w_{l,1}'w_{l+1,2}w_{l,3}' = (w_{l+1,1}'w_{l+1,2}'u)w_{r,2}'w_{r,3}',
\end{split}
\end{equation}
Since $0<|w_{l,3}'|-|w_{l-1,3}'|<n$, we have that
$w_{l,2}$ and $w_{l,2}'$ overlap in $w_{l}$.
Since $w_{l+1,2}$ and $w_{l+1,2}'$ overlap in
$w_{l+1}$, we have that $|w_{l,3}'|>|w_{r,3}'|$.
By the choice of $l$, $|w_{l,3}'|-|w_{r,3}'|<n$.
Hence $w_{l+1,2}$ and $w_{r,2}'$ overlap in
$w_{l+1}$. Thus, applying the Lemma~\ref{L1} to
(\ref{xeq10aa}), we obtain that $w_{l}=w_{l+1}$.
Now the sequence $w_0,w_1,\dots
,w_{l},w_{l+2},w_{l+3},\dots ,w_t$ with the
decomposition as in (\ref{xeq5bis}), except for
$w_{l}=w_{l-1,1}'w_{l,2}w_{l-1,3}' =
w_{l+1,1}'w_{l+1,2}'w_{l+1,3}'$, is a shorter
sequence with a decomposition of type
(\ref{xeq5bis}), in contradiction with the
minimality of $t$. Hence
$|w_{r,3}'|>|w_{r-1,3}'|$. Therefore
$|w_{r,3}'|-|w_{r-1,3}'|\geq n$.

Since
$w_r=w_{r-1,1}'w_{r,2}w_{r-1,3}'=w_{r,1}'w_{r,2}'w_{r,3}'$,
 we thus have that $|w_{r-1,1}'|=
|w_{r,3}'|-|w_{r-1,3}'|+|w_{r,1}'| \geq
n+|w_{r,1}'|$ and therefore
\begin{eqnarray}
\label{sixth*} w_{r-1,1}'\in
w_{r,1}'w_{r,2}'\mathrm{FM}_n.
\end{eqnarray}
Since $0<|w_{0,3}'|<|w_{1,3}'|<\dots <
|w_{r-1,3}'|$ and
$$w_k=w_{k-1,1}'w_{k,2}w_{k-1,3}'=w_{k,1}'w_{k,2}'w_{k,3}',$$
for all $k=1,2,\dots , r-1$, we have that
$w_{k-1,1}'\in w_{k,1}'\mathrm{FM}_n$, for all
$k=1,2,\dots , r-1$. Thus,  from
(\ref{sixth*}), for all $k=0,1,2,\ldots , r-1$,
there exists $v_k'\in \mathrm{FM}_n$ such that
$$w_{k,1}'=w_{r,1}'w_{r,2}'v_k'.$$

Consider the following sequence:

\begin{equation} \label{xeq14aa}
\begin{split}
 & w_0=w_{r,1}'w_{r,2}'(v_0'w_{0,2}'w_{0,3}'), \\
 &w_1'=w_{r,1}'w_{r+1,2}(v_0'w_{0,2}'w_{0,3}')=(w_{r,1}'w_{r+1,2}v_0')w_{0,2}'w_{0,3}', \\
 &w_2'=(w_{r,1}'w_{r+1,2}v_0')w_{1,2}w_{0,3}'=(w_{r,1}'w_{r+1,2}v_1')w_{1,2}'w_{1,3}', \\
 &\hspace{0.2 cm} \vdots   \\
 &w_{r-1}'=(w_{r,1}'w_{r+1,2}v_{r-3}')w_{r-2,2}w_{r-3,3}'=(w_{r,1}'w_{r+1,2}v_{r-2}')w_{r-2,2}'w_{r-2,3}', \\
 &w_r'=(w_{r,1}'w_{r+1,2}v_{r-2}')w_{r-1,2}w_{r-2,3}'=(w_{r,1}'w_{r+1,2}v_{r-1}')w_{r-1,2}'w_{r-1,3}', \\
 &w_{r+1}'=(w_{r,1}'w_{r+1,2}v_{r-1}')w_{r,2}w_{r-1,3}'.
\end{split}
\end{equation}
Since
$w_r=w_{r-1,1}'w_{r,2}w_{r-1,3}'=w_{r,1}'w_{r,2}'w_{r,3}'$
and $w_{r-1,1}'=w_{r,1}'w_{r,2}'v_{r-1}'$, we
have that $w_{r,3}'=v_{r-1}'w_{r,2}w_{r-1,3}'$.
Hence
$$w_{r+1}=w_{r,1}'w_{r+1,2}w_{r,3}'=w_{r,1}'w_{r+1,2}v_{r-1}'w_{r,2}w_{r-1,3}'=w_{r+1}'.$$
Note that if $w_1'=x_{k_1}\cdots x_{k_m}$, then
$x_{k_m}=x_{0(m)}=x_{t(m)}$ and
$$\pi(x_{k_1}\cdots
x_{k_{m-1}})=\pi(x_{0(1)}\cdots x_{0(m-1)})\neq
\pi(x_{t(1)}\cdots x_{t(m-1)}).$$ Now the
sequence $w_1',w_2',\dots
,w_r',w_{r+1},w_{r+2},\dots ,w_t$, with the
decomposition (\ref{xeq14aa}) for
$w_1',w_2',\dots ,w_r'$, the decomposition
$$w_{r+1}=(w_{r,1}'w_{r+1,2}v_{r-1}')w_{r,2}w_{r-1,3}'=w_{r+1,1}'w_{r+1,2}'w_{r+1,3}'$$
 for $w_{r-1}$, and the decomposition
(\ref{xeq5bis}) for $w_{r+2},\dots ,w_t$, is a
shorter sequence with a decomposition of type
(\ref{xeq5bis}), in contradiction with the
minimality of $t$. Therefore the claim follows.

In particular we have that $|w'_{j,3}|>0$ for all
$j=1,\dots, t-1$. Hence
$x_{0(m)}=x_{t(m)}=x_{1(m)}$. Now,
$w_0=w_{0,1}'w_{0,2}'$ and $w_1=w_{0,1}'w_{1,2}$,
with $w_{0,2}'\neq w_{1,2}$, by the minimality of
$t$. Since $\pi(w_{0,2}')=\pi(w_{1,2})=z$, there
exist  different $\sigma,\tau\in H$ such that
 $$w_{0,2}'=x_{\sigma(1)}\cdots
x_{\sigma(n)}\quad\mbox{and}\quad
w_{1,2}=x_{\tau(1)}\cdots x_{\tau(n)}.$$ Since
$x_{\sigma(n)}=x_{0(m)}=x_{t(m)}=x_{1(m)}=x_{\tau(n)}$,
 we have that $\id\neq
\sigma^{-1}\tau\in H_n$. But  this yields a contradiction as, by
assumption, $H_n=\{ \id\}$. Therefore $S$ is right cancellative.
\end{proof}

\section{The finitely presented algebra $K[S_n(H)]$}

We begin with some properties of prime ideals.

Let $H$ be a subgroup of $\Sym_n$. Recall that $z=a_1a_2\cdots
a_n\in S_n(H)$. In \cite{alghomshort} it is proved that if
$H=\langle (1,2,\dots ,n)\rangle$  and $n\geq 3$, then $z$ is a
central element of $S_n(H)$ and $zS_n(H)$ is a minimal prime ideal
of $S_n(H)$.

We shall see that, for an arbitrary  abelian subgroup $H$ of
$\Sym_n$ the behaviour is different. Indeed we show that
$S_n(H)zS_n(H)$ is a prime ideal of $S_n(H)$,  for $n\geq 3$,
but it is not minimal in general.

 First we shall see that $S_n(H)zS_n(H)$ is a prime ideal of
$S_n(H)$ for an arbitrary non-transitive subgroup $H$ of $\Sym_n$.

\begin{lemma}\label{nontrans}
If $H$ is a non-transitive subgroup of $\Sym_n$, then
$S_n(H)zS_n(H)$ is a prime ideal of $S_n(H)$.
\end{lemma}

\begin{proof}
Let $u,v\in S_n(H)\setminus S_n(H)zS_n(H)$. Since
$H$ is not transitive, there exist $1\leq i,
j\leq n$ such that $i\neq\sigma(1)$ and $j\neq
\sigma(n)$, for all $\sigma\in H$. It is then
clear that $ua_j^2a_i^2v\notin S_n(H)zS_n(H)$.
Thus $S_n(H)zS_n(H)$ is prime.
\end{proof}

Recall that a subgroup $H$ of $\Sym_n$ is semiregular if $H_i=\{
\id\}$ for all $1\leq i\leq n$.

\begin{lemma}\label{L3}
If $H$ is an abelian subgroup of $\Sym_n$ and $S=S_n(H)$, with
$n\geq 3$, then $SzS$ is a prime ideal of $S$.
\end{lemma}

\begin{proof}
Let $u,v\in S\setminus SzS$.

 By Lemma~\ref{nontrans}, we may assume that $H$ is a transitive
subgroup of $\Sym_n$.  Because, by assumption, $H$ is abelian, by
\cite[Proposition 3.2]{passman} we  then have that $H$ is
semiregular. Therefore $n=|\{\sigma (1)\mid \sigma\in
H\}|=|H|/|H_1|=|H|$. By the   comment before the
Lemma~\ref{nontrans}, we may assume that $(1,2,\dots ,n)\notin H$.
Let $i,j$ be such that  $u\in Sa_i\cup \{ 1\}$ and $v\in a_jS\cup
\{1\}$. By Lemma~\ref{L2}, we have  that $ua_i\notin Sz$ and
$a_jv\notin zS$. Hence, since $n\geq 3$, we have that
$ua_i^2a_j^2v\notin SzS$. Therefore $SzS$ is a prime ideal of $S$
and the lemma follows.
\end{proof}

\begin{lemma}\label{AA}
Let $H$ be a  subgroup  of $\Sym_n$ such that $H_2=H_{n-1}=\{
\id\}$,  with $n\geq 3$. Then, for all $1\leq i\leq n$, there exist
$1\leq j,j'\leq n$ such that $j\neq i$, $j'\neq i$, $x_ix_j\notin A$
and $x_{j'}x_i\notin \tilde{A}$,  where $A$ and $\tilde{A}$ are
defined in (\ref{third*}).
\end{lemma}

\begin{proof}
Suppose that $\{ x_ix_j\mid 1\leq j\leq n,\; j\neq i\}\subseteq A$.
Since $n\geq 3$, there exist $1\leq j,k\leq n$ such that $i,j,k$ are
three different integers. Because $x_ix_j,x_ix_k\in A$, there exist
$\sigma,\tau\in H$ such that $\sigma(n-1)=i=\tau(n-1)$,
$\sigma(n)=j$ and $\tau(n)=k$. As $H_{n-1}=\{\id\}$  and
$\sigma(n-1)=\tau(n-1)$, we have that $\sigma=\tau$.  But this
contradicts with   $\sigma(n)=j\neq k=\tau(n)$. Therefore there
exists $1\leq j\leq n$ such that $j\neq i$ and $x_ix_j\notin A$.

Similarly  one proves that there exists
$1\leq j'\leq n$ such that $j'\neq i$ and
$x_{j'}x_i\notin\tilde{A}$.
\end{proof}

 \begin{lemma}\label{AB}  Let $H$ be a transitive subgroup of $\Sym_n$, and
let $S=S_n(H)$. Then $\cup_{i=1}^n S{a_i}^2S$ is not a prime ideal
in $S$.
\end{lemma}

\begin{proof}  Let $Q=\cup_{i=1}^n S{a_i}^2S$.
Note that $z \notin Q$. However, since $H$ is transitive, we have
that $za_{i}\in Q$ for all $i$. Hence $zSz\subseteq Q$ and therefore
$Q$ is not prime.
\end{proof}

 Now we shall see another general result on prime ideals of
$S_n(H)$ for an arbitrary subset $H$ of $\Sym_n$.

\begin{theorem}\label{L4}
Let $H$ be a subset  of $\Sym_n$,  and let $K$ be a field. If $Q$ is
a prime ideal in $S_n(H)$ such that $S_n(H)zS_n(H)\subseteq Q$, then
$K[S_n(H)]/K[Q]$ is a prime monomial algebra. Furthermore, if $Q$ is
finitely generated then $K[S_n(H)]/K[Q]$ is either PI or primitive.
\end{theorem}

\begin{proof}
Let $S=S_n(H)$. Let $\mathrm{FM}_n= \langle x_1, \ldots , x_n
\rangle $ be the free monoid of rank $n$, and let $ \pi\colon
\mathrm{FM}_n \rightarrow S $ be the unique morphism such that $
\pi(x_i) = a_i $ for all $ i = 1, \ldots ,n $. Note that
$$\pi^{-1}(SzS)=\mathrm{FM}_n x_1 x_2\cdots x_n \mathrm{FM}_n
\cup \bigcup_{\sigma\in H}\mathrm{FM}_n x_{\sigma(1)}
x_{\sigma(2)}\cdots x_{\sigma(n)} \mathrm{FM}_n.$$  Thus
$\mathrm{FM}_n / \pi^{-1}(SzS) \cong S/SzS$ and $K[S]/K[SzS]$ is a
monomial algebra. Since $SzS\subseteq Q$, we have that $S/Q\cong
\mathrm{FM}_n/\pi^{-1}(Q)$. Hence $K[S]/K[Q]$  is a monomial
algebra. Since $Q$ is a prime ideal of $S$, by \cite[Proposition
24.2]{book}, $K[S]/K[Q]$ is prime.

Suppose that $Q$ is finitely generated. Then $K[S]/K[Q]$ is a
finitely presented monomial algebra, and by
\cite[Theorem~1.2]{BellColak} this algebra is either PI or
primitive.
\end{proof}

 We have seen in Lemma~\ref{L3} that, for an arbitrary abelian
subgroup $H$ of $\Sym_n$,  $S_n(H)zS_n(H)$ is a prime ideal of
$S_n(H)$,  for $n\geq 3$. The following result shows that it is not
minimal in general.

\begin{theorem}\label{T1}
Let $n\geq 3$. Let $H$ be either a non-transitive subgroup of
$\Sym_n$ or an abelian subgroup of $\Sym_n$, such that $(1,2,\dots
,n)\notin H$. Let $S=S_n(H)$ and let $r$ be a positive integer.
\begin{itemize}
\item[(i)] If
 $m_1,\dots ,m_r \geq 3$ and $1\leq i_1,\dots ,i_r\leq n$, then
$\cup_{j=1}^r S{a_{i_j}}^{m_j}S$ is a prime ideal in $S$.
\item[(ii)] For $m \geq 1$, $Sz^mS$ is a prime ideal in $S$.
\item[(iii)] If
 $m_1,\dots ,m_r \geq 3$, $1\leq i_1,\dots ,i_r\leq n$ and $m\geq 1$, then
$(Sz^mS)\; \cup \;\bigcup_{j=1}^r
S{a_{i_j}}^{m_j}S$ is a prime ideal in $S$.
\end{itemize}
\end{theorem}

\begin{proof}

 We rely on the fact that $z$ never involves letters to a power
$\geq 2$.

$(i)$ Let $Q = \cup_{j=1}^r S{a_{i_j}}^{m_j}S$. Let $w_1, w_2 \in S
\setminus Q$. We shall see that $w_1Sw_2\not\subseteq Q$. We may
assume that $w_1\neq 1$ and $w_2\neq 1$.

{\em Case(A):} $H$ is not transitive. Thus there exists
$l$, with $1 \leq l < n$, such that $l$ is not in the
orbit of $n$, and there exists $l'$, with $1< l' \leq n$, such
that $l'$ is not in the orbit of $1$. In the event that
$l=l'$, let $l''$, with $1 \leq l'' \leq n$, be such
that $l'' \neq l=l'$.

If $w_1 \in S \setminus Sa_l$, let $w_1'= w_1 a_l^2$.
If $w_1 \in Sa_l \setminus Sa_l^2$, let
$w_1'=w_1 a_l$.
If $w_1 \in Sa_l^2$, let $w_1'=w_1$.
If $w_2 \in S \setminus a_{l'}S$, let $w_2'=a_{l'}^2w_2$.
If  $w_2 \in a_{l'}S \setminus a_{l'}^2S$, let
$w_2'=a_{l'}w_2$.
If  $w_2 \in a_{l'}^2S$, let $w_2'=w_2$.
In the event that $l \neq l'$, we have that
$w_1'w_2' \notin Q$, otherwise $w_1'a_{l''}w_2' \notin Q$.

{\em Case(B):} $H$ is an abelian subgroup of $\Sym_n$, such that
$(1,2,\dots ,n)\notin H$.

In this case, we may assume that $H$ is transitive.
By \cite[Proposition 3.2]{passman} we then have that $H$ is
semiregular. Thus $H_1=H_2=H_{n-1}=H_n=\{\mbox{id}\}$.

Suppose that $w_1\in Sa_k$ and $w_2\in a_lS$. By
Lemma~\ref{AA}, there exist $j,j'$ with $1\leq j,j'\leq n$
such that $j\neq k$, $j'\neq l$, $x_kx_j\notin A$
and $x_{j'}x_l\notin \tilde{A}$. If $j\neq j'$,
then, by Lemma~\ref{L2},
$w_1a_j^2a_{j'}^2w_2\notin Q$. If $j=j'$, then by
Lemma~\ref{L2}, $w_1a_ja_{j'}w_2\notin Q$.

$(ii)$  Let $w_1, w_2 \in S \setminus Sz^mS$. We shall see that
$w_1Sw_2\not\subseteq Sz^mS$.

{\em Case(A):} $H$ is not transitive. Let $l,l'$ be as in the
proof of $(i)$. Then $w_1a_l^2a_{l'}^2w_2 \notin Sz^mS$.

{\em Case(B):} $H$ is an abelian subgroup of $\Sym_n$, such that
$(1,2,\dots ,n)\notin H$. This is proved similarly as $(i)$.

$(iii)$  Let $Q = \cup_{j=1}^r S{a_{i_j}}^{m_j}S$. Let $w_1, w_2 \in
S \setminus (Sz^mS\cup Q)$. By an argument similar to the one
used in the proof of $(i)$, one can prove that $w_1Sw_2\not\subseteq
Sz^mS\cup Q$.
\end{proof}

\begin{corollary}
Let $n\geq 3$. Let $H$ be either a non-transitive subgroup of
$\Sym_n$ or an abelian subgroup of $\Sym_n$, such that $(1,2,\dots
,n)\notin H$. Let $S=S_n(H)$ and let $K$ be a field. For $m_1,\dots
,m_r \geq 3$ and $1\leq i_1,\dots ,i_r\leq n$, let $Q=\cup_{j=1}^r
S{a_{i_j}}^{m_j}S$. Then $K[S]/K[SzS\cup Q]$ is a primitive monomial
algebra.
\end{corollary}

\begin{proof}
By Theorem~\ref{T1}, $SzS\cup Q$ is a prime ideal
of $S$. Hence by Theorem~\ref{L4},
$K[S]/K[SzS\cup Q]$ is a finitely presented and
prime monomial algebra, and by
\cite[Theorem~1.2]{BellColak}, it is either PI or
primitive. Note that the submonoid $\langle
a_1a_2,a_1a_3\rangle$ of $S$ is a free monoid of
rank two and $$\langle a_1a_2,a_1a_3\rangle\cap
(SzS\cup Q)=\emptyset.$$ Hence $K[S]/K[SzS\cup
Q]$ is not PI. Therefore $K[S]/K[SzS\cup Q]$ is
primitive.
\end{proof}

Although it is well-known that in a commutative semigroup the union
of prime ideals is prime, this is not true for noncommutative
semigroups. Thus part $(iii)$ of Theorem~\ref{T1} is not a trivial
consequence of parts $(i)$ and $(ii)$. In fact we have the following
result.

\begin{proposition}
Let $H$ be a transitive abelian subgroup of $\Sym_n$, with $n>2$,
such that $(1,2,...,n)\notin H$. Let $S=S_n(H)$. Then there exist
prime ideals $P,Q$ in $S$ such that $P\cup Q$ is not a prime ideal
in $S$.
\end{proposition}

\begin{proof}
Let $P_i=Sa_i^2S$, for $1\leq i\leq n$. By Lemma~\ref{AB},
$\cup_{i=1}^{n} P_i$ is not a prime ideal in S. We shall prove that
each $P_i$ is a prime ideal in $S$.

Let $w_1,w_2\in S\setminus P_i$. We shall see that
$w_1Sw_2\not\subseteq P_i$. We may assume that $w_1\neq 1$ and
$w_2\neq 1$. Suppose that $w_1\in Sa_k$ and $w_2\in a_lS$. Since $H$
is transitive and abelian, by \cite[Proposition 3.2]{passman}, $H$
is semiregular. Thus by Lemma ~\ref{AA}, there exist $j,j'$, with
$1\leq j,j' \leq n$, sucht that $k\neq j$, $l\neq j'$, $x_kx_j\notin
A$ and $x_{j'} x_l\notin \tilde{A}$. If $j\neq i$ and $j'\neq i$,
then $w_1a_j^2a_{j'}^2w_2\notin P_i$, by Lemma~\ref{L2}. If $j\neq
i$ and $j'=i$, then $l\neq i$ and $w_1a_j^2a_l^2w_2\notin P_i$, by
Lemma~\ref{L2}. If $j= i$ and $j'\neq i$, then $k\neq i$ and
$w_1a_k^2a_{j'}^2w_2\notin P_i$, by Lemma~\ref{L2}. If $j= i=j'$,
then $k\neq i$, $l\neq i$ and $w_1a_k^2a_l^2w_2\notin P_i$, by
Lemma~\ref{L2}. Therefore the result follows.
\end{proof}

We finish with handling the Jacobson radical of $K[S_{n}(H)$
for $H$ abelian.

In \cite[Corollary 2.2]{altalgebra} it is proved
that if $H$ is an arbitrary subgroup of $\Sym_n$
and the Jacobson radical $J(K[S_n(H)])\neq \{
0\}$, then $H$ is a transitive subgroup of
$\Sym_n$.

In \cite{alghomshort} it is proved that if
$H=\langle (1,2,\dots ,n)\rangle$ then
$J(K[S_n(H)])=\{ 0\}$. Now we generalize this
result for any abelian subgroup $H$ of $\Sym_n$.

Recall that if $\alpha =\sum_{s\in S_{n}(H)} k_{s}s$, with $k_{s}\in
K$, then by $\supp (\alpha )$ one denotes the support of $\alpha$.
That is, $\supp (\alpha )=\{ s\in S \mid k_{s}\neq 0\}$.

\begin{theorem}\label{T2}  If  $H$ is an abelian subgroup of $\Sym_n$
then $J(K[S_n(H)])=\{ 0\}$.
\end{theorem}

\begin{proof}
Suppose $H$ is an abelian subgroup of $\Sym_n$. Let $S=S_{n}(H)$.
Note that for $n\leq 2$, $K[S]$ is either a polynomial algebra over
$K$ or a free algebra over $K$. Thus we may assume that $n\geq 3$.

We prove the result by contradiction. So, assume $0\neq \alpha
=\sum_{s\in S} k_{s}s\in J(K[S_n(H)])$, with each $k_{s}\in K$.
Hence, by the comments before the Theorem, $H\neq \langle
(1,2,\ldots , n)\rangle$ and $H$ is a transitive abelian subgroup of
$\Sym_{n}$. Thus, as mentioned before (\cite[Proposition
3.2]{passman}), $H$ is semiregular. Therefore $n=|\{\sigma (1)\mid
\sigma\in H\}|=|H|/|H_1|=|H|$. So, $(1,2,\dots ,n)\notin H$. Now,
 since $n\geq3$, from  \cite[Proposition 2.6]{altalgebra}, we
know that $J(K[S]) \subseteq [Sz \cup zS]$.  Let $w \in \supp
(\alpha)$. Then $w \in {Sz \cup zS}$ and  $w\in a_iS\cap Sa_j$, for
some $i,j$. By Lemma~\ref{L2}, $a_iwa_j \notin {Sz \cup zS}$. Since
$a_i \alpha a_j \in J(K[S])$, there exists $w'\in \supp (\alpha)$
such that $w\neq w'$ and $a_iwa_j=a_iw'a_j$. However,  from
Theorem~\ref{T3} we know that $S$ is cancellative, and thus $w=w'$,
a contradiction.
\end{proof}

\vspace{30pt}
 \noindent \begin{tabular}{llllllll}
 F. Ced\'o && E. Jespers  \\
 Departament de Matem\`atiques &&  Department of Mathematics \\
 Universitat Aut\`onoma de Barcelona &&  Vrije Universiteit Brussel  \\
08193 Bellaterra (Barcelona), Spain    && Pleinlaan
2, 1050 Brussel, Belgium \\
 cedo@mat.uab.cat && efjesper@vub.ac.be\\
   &&   \\
G. Klein &&  \\ Department of Mathematics &&
\\  Vrije Universiteit Brussel && \\
Pleinlaan 2, 1050 Brussel, Belgium &&\\ gklein@vub.ac.be&&
\end{tabular}

\begin{thebibliography}{99}
\itemsep=-2pt
\bibitem{BellColak} J.P. Bell, P. Colak,
Primitivity of finitely presented monomia
algebras, J. Pure App. Algebra 213 (2009),
1299--1305.
\bibitem{alghomshort} F. Ced\'o, E. Jespers and J. Okni\'nski,
Finitely presented algebras and groups defined by
permutation relations, J. Pure App. Algebra 214
(2010), 1095--1102.
\bibitem{altalgebra} F. Ced\'o, E. Jespers and J. Okni\'nski,
Algebras and groups defined by permutation
relations of alternating type, preprint.
\bibitem{alt4algebra} F. Ced\'o, E. Jespers and J. Okni\'nski,
The radical of the four generated algebra of the
alternating type, Contemporary Math 499 (2009),
1--26.
\bibitem{book} J. Okni\'nski,
Semigroup Algebras, Marcel Dekker, New York,
1991.
\bibitem{passman} D. S. Passman, {\em Permutation Groups}, Benjamin, New York, 1968.

\end{thebibliography}
\end{document}